\documentclass[11pt]{amsart}

\usepackage{amsmath,amsfonts,amssymb,amsthm,enumerate,tikz-cd,color}

\def\Z{\mathbb Z}
\def\Q{\mathbb Q}
\def\F{\mathbb F}

\newcommand{\Sel}{\mathrm{Sel}}
\newcommand{\Lie}{\mathrm{Lie}}
\newcommand{\an}{\mathrm{an}}
\newcommand{\sat}{\mathrm{sat}}
\newcommand{\AJ}{\mathrm{AJ}}

\theoremstyle{plain}
\newtheorem{theorem}{Theorem}
\newtheorem{conjecture}{Conjecture}
\newtheorem{lemma}{Lemma}
\newtheorem{proposition}{Proposition}

\theoremstyle{definition}

\newtheorem{remark}{Remark}

\DeclareMathOperator{\Gal}{Gal}
\DeclareMathOperator{\dR}{dR}

\newcommand{\Res}{\mathrm{Res}}

\newcommand{\Pic}{\mathrm{Pic}}
\DeclareMathOperator{\loc}{loc}

\newcommand{\Ker}{\mathrm{Ker}}

\DeclareMathOperator{\rk}{rk}

\newcommand{\Jac}{\mathrm{Jac}}

\DeclareMathOperator{\C}{\mathbb{C}}
\begin{document}
\title{The Chabauty--Coleman method and $p$-adic linear forms in logarithms}
\author{Netan Dogra}
\maketitle

\setcounter{tocdepth}{1}
\newcommand{\todo}[1]{{\color{red}To do: {#1}}}
\pagestyle{headings}
\markright{CHABAUTY--COLEMAN AND LINEAR FORMS IN LOGARITHMS}
\begin{abstract}
Results in $p$-adic transcendence theory are applied to two problems in the Chabauty--Coleman method. The first is a question of McCallum and Poonen regarding repeated roots of Coleman integrals. The second is to give lower bounds on the $p$-adic distance between rational points in terms of the heights of a set of Mordell--Weil generators of the Jacobian. We also explain how, in some cases, a conjecture on the `Wieferich statistics' of Jacobians of curves implies a bound on the height of rational points of curves of small rank, in terms of the usual invariants of the curve and the height of Mordell--Weil generators of its Jacobian. The proof uses the Chabauty--Coleman method, together with effective methods in transcendence theory. We also discuss generalisations to the Chabauty--Kim method.
\end{abstract}

\tableofcontents
\section{Introduction}

\subsection{Main results}
Let $X/\Q $ be a smooth projective geometrically irreducible curve of genus $g>1$, and let $J$ denote its Jacobian. The Chabauty--Coleman method is an approach to determining the set $X(\Q )$ of rational points \cite{chabauty},\cite{coleman}. It depends on the assumption (henceforth referred to as the \textit{Chabauty--Coleman hypothesis}) that the topological closure of $J(\Q )$ in $J(\Q _p )$, which we will denote by $\overline{J(\Q )}$, is not finite index. This will happen, for instance, whenever the rank of $J$ is less than $g$. In this case, we have a non-zero subspace $W\subset H^0 (X_{\Q _p },\Omega )$, called the space of \textit{vanishing differentials}, defined to be the annihilator of the image of $J(\Q )\otimes \Q _p$ in $H^0 (X_{\Q _p },\Omega )^* $ under the $p$-adic logarithm map
\[
\log _p :J(\Q _p )\to \Lie (J)_{\Q _p }\simeq H^0 (X_{\Q _p },\Omega )^* .
\]
Then $X(\Q )$ is contained in the finite set
\[
X(\Q _p )_1 :=\{x\in X(\Q _p ):\forall \omega \in W,\int ^x _b \omega =0 \},
\]
where $b$ is a point in $X(\Q )$.

Our first result concerns a question of McCallum and Poonen \cite[\S 7, Problem 4]{mccallumpoonen} on the effectivity of the Chabauty--Coleman algorithm. To explain this question, we briefly recall how the Chabauty--Coleman method may be used to determine the set of rational points. Under the hypothesis above, $X(\Q _p )_1 $ is a finite set containing $X(\Q )$, and for any $n$ one can compute its image in $X(\mathbb{Z}/p^n \mathbb{Z})$ by computing $p$-adic approximations to $\int _b \omega $ for $\omega $ in a basis of $W$.

One can compute the zeroes of $\int _b \omega $ by $p$-adically approximating the first few terms of its power series expansion on each residue disk, and bounding the valuation of the coefficients to estimate the roots. The first obstacle to using this to determine $X(\Q )$ is if there is a point $x\in X(\mathbb{Z}/p^n \mathbb{Z})$ which is in the image of $X(\Q _p )_1 $ but \textit{not} in the image of $X(\Q )$. In thise case, one may hope to prove that $x$ does not come from $X(\Q )$ using the Mordell--Weil sieve, see for example \cite{bruinstoll}, \cite{siksek:sieve}.

However, there is also a second problem if there is a $p$-adic point $x\in X(\Q _p )$ such that, for all vanishing differentials $\omega $, the function $\int _b \omega $ has a repeated root  at $x$. Unless one has a (theoretical) way of proving that $x$ is a repeated root, it is impossible to rule out that $\int _b \omega $ has two zeroes very close to $b$. If $x$ is not a rational point, one could hope to rule out both of these points being rational using the Mordell--Weil sieve, but if $x$ is in fact a rational points this will not work. Hence one needs a theoretical rather than computational understanding of the rational roots of $\omega $. Our first result gives a classification of such roots in the simplest nontrivial case of Chabauty's method (where all simple factors of the Jacobian have rank at most one or have rational points which are dense in a finite index subgroup of the $p$-adic points).

\begin{theorem}\label{thm1}
Suppose that $\overline{J(\Q )}$ is not finite index in $J(\Q _p )$, and suppose that for every simple isogeny factor $A$ of $J$, either $A$ has Mordell--Weil rank at most one or $\overline{A(\Q )}$ has finite index in $A(\Q _p )$. Then (at least) one of the following holds.
\begin{enumerate}
\item For any algebraic point $x$ of $X$, there is a vanishing differential $\omega $ such that the Coleman integral $\int _b \omega $ does not have a repeated root at $x$.
\item There is a cover $f:X\to Y$ such that $x$ is a ramification point of $f$, and $Y$ is either a rank zero elliptic curve or a curve of genus greater than 1 which also satisfies the Chabauty--Coleman hypothesis, and has the property that at any algebraic point $x$ of $X$, there is a vanishing differential $\omega $ such that the Coleman integral $\int \omega $ does not have a repeated root at $x$.
\end{enumerate}
\end{theorem}
\begin{remark}
In case (2), if the identity component $P(Y/X)$ of $\Ker (J\to \Jac (Y))$ has the property that the topological closure of $P(Y/X)(\Q )$ in $P(Y/X)(\Q _p )$ has finite index in $P(Y/X)(\Q _p )$, then case (1) does not occur. In general (1) and (2) need not be mutually exclusive.
\end{remark}
The second result concerns the \textit{$p$-adic proximity} of points of $X(\Q _p )_1$ to points in $X(\Q )$. For points $x,y\in X(\Q _p )$, we define the $p$-adic distance between $x$ and $y$ to be $p^{-d}$, where $d$ is the largest $m$ such that $x$ and $y$ have the same image in $X(\mathbb{Z}/p^m \mathbb{Z})$, for $X/\mathbb{Z}$ a minimal proper regular model. We suppose we are given a point $x\in X(\Q )$, and want to bound the proximity of other points in $X(\Q _p )_1$ to $x$ in terms of the height of $x$ and of a set of Mordell--Weil generators of the Jacobian. Let $X$ be a curve of genus $g>1$ with Jacobian $J$, and fix a projective embedding of $J$, defining a (logarithmic) height $h$.
\begin{theorem}\label{thm2}
Suppose that $\overline{J(\Q )}$ is not finite index in $J(\Q _p )$. Let $\iota _i :A_i \hookrightarrow J$ be simple abelian subvarieties of $J$ of degree $d_i$ such that
\[
\prod _{i=1}^n A_i \to J
\]
is an isogeny. Suppose that for all $i$, either $A_i $ has Mordell--Weil rank at most one or $\overline{A_i (\Q )}$ has finite index in $A_i (\Q _p )$. Then there is an effectively computable constant $c_L $ depending on a choice $L$ of $\mathbb{Z}$-basis of $H^1 (X_{\mathbb{Z}},\mathcal{O}_X )$ (where $X_\mathbb{Z}$ is a minimal proper regular model of $X$) with the following property. For any $x\in X(\Q )$ and $y\in X(\Q _p )_1 $ distinct from $x$, if $x$ is congruent to $y$ modulo $p^N$ then
\begin{equation}\label{eqn:main_ineq}
N <c_0 c_L (h(x)+c_1 \kappa )(m_p h_0 +c_2 )^n (\log (h(x)+c_1 \kappa ) +\log (m_p h_0 +c_2 )+c_3 )^{n+3},
\end{equation}
where $m_p $ is the number of $\F _p$ points of the N\'eron model of $J$,  $P_1 ,\ldots P_r$ in $J(\Q )$ are a set of Mordell--Weil generators, 
$h_0$ is $\max \{\log (3),h(P_1 ),\ldots ,h(P_r )\}$, $\kappa :=\max \{1,h_F (J),\log (d_i )\}$, $h(x)$ is the height of $x$ in $\mathbb{P}^{g-1}$ with respect to the canonical embedding and the basis of $H^1 (X,\mathcal{O})$ above, $h_F (J)$ is the Faltings height of $J$, and $c_0 ,c_1 ,c_2 ,c_3 $ are constants depending only on $g$.
\end{theorem}

\begin{remark}
The reason that Theorems \ref{thm1} and \ref{thm2} have the restriction to simple factors of the Jacobian having rank 1 is as follows. If $A$ is a simple abelian variety over $\Q $ and $P$ is a rational point of $A$ of infinite order, then results in $p$-adic transcendence theory due to Fuchs and Pham (which are $p$-adic analogues of theorems of W\"ustholz \cite{wustholz} and Hirata--Kohno \cite{HK1} respectively) imply that the $p$-adic logarithm of $P$ can never be contained in a proper $\Q $-subspace of the Lie algebra of $A$, and control how close $\log (P)$ can get to such subspaces. As we explain below, the settings of Theorems \ref{thm1} and \ref{thm2} instead concern the reverse inclusion: i.e. situations when a proper $\Q _p$ subspace of $\Lie (A)_{\Q _p }$ generated by the logarithms of algebraic points contains (or nearly contains) a line defined over $\Q $. Conjectures in transcendence theory imply this can never happen, but as far as we know nothing has been proved about this. Hence we restrict to the case that the rank of $A$ is one, so the inclusion is an equality and can be reversed.

Note that, by a recent theorem of R\'emond and Gaudron we may the integers $d_i$ appearing in the statement of Theorem \ref{thm2} in terms of the Faltings height of $J$ \cite[Th\'eor\`eme 1.3]{remond2020nouveaux}. The reason that we have phrased \eqref{eqn:main_ineq} as we have is that, as will be apparent from the proof, the constant $c_L$ is a constant $\omega _L ^{g+3}$ which arises in the theorem of Fuchs and Pham \cite[Theorem 2.1]{fuchs_pham}, and hence any bound for that constant gives a more explicit form of Theorem \ref{thm2}. In particular, if one had a $p$-adic analogue of the theorem of Gaudron \cite{gaudron} bounding the proximity of archimedean logarithms of rational points to rational subspaces, then one could bound $N$ purely in terms of explicit invariants of $X,J$ and $x$.
\end{remark}

What does controlling the $p$-adic proximity tell us about the \textit{height} of rational points on $X$? As we explain below, this question is closely related to the problem of understanding the \textit{Wieferich statistics} of the Jacobian of $X$. More precisely, if $Q\in J(\Q )$ denotes a point of infinite order on $J(\Q )$, then we can prove an effective height bound on the rational points of $X$ in terms of the regulator of the Jacobian of $X$ if we assume there is an effectively computable constant $p_0$ such that for all $p>p_0$, $\log _p (Q )$ does not vanish modulo $p$.

If we instead ask for a bound on the height of $S$-integral points (for some finite set of primes $S$, and relative to some zero-dimensional subscheme of $X$) then these methods do give a bound in terms of the regulator of the Jacobian (still with the restriction that the genus is two). As, in this case, such a bound has also been proved by Hirata--Kohno \cite{HK2}, without any condition on the genus or the Mordell--Weil rank (but still in terms of the regulator), we do not pursue this here.
%

%
\subsection{Chabauty--Coleman and Wieferich statistics}
A \textit{Wieferich prime} is a prime $p$ such that $2^{p-1}$ is congruent to 1 modulo $p^2$. Equivalently, it is a prime such the $p$-adic logarithm of 2 vanishes to first order. In \cite{katz}, Katz considered the more general problem of determining the statistical properties of the mod $p$ logarithm of a point $Q$ on a $g$-dimensional semi-abelian variety $G$ (over $\mathbb{Z}[1/N]$). He explained how to define a $g$-dimensional torus $T=\Lie (G)\otimes \mathbb{R}/\mathbb{Z}$ such that $\log _p (Q)$ may be viewed as element of $T$ for almost all $p$. Katz's conjecture \cite{katz}[Conjecture 13.1] in particular states that the sequence $(\log _p (Q))$ is \textit{equidistributed} in $T$ if $\langle Q \rangle $ is Zariski dense in $G$.

We show that a slightly different randomness property of the Wieferich statistics of the mod $p$ logarithms of rational points on semi-abelian varieties would give new cases where one can bound the height of rational points on curves in terms of the regulators of their Jacobians. The randomness property we need is that the mod $p$ logarithm of a point lies in a proper two dimensional subspace of $\Lie (G)$ for only finitely many $p$, unless there is an `algebraic reason' for it do so. This is motivated by the heuristic that, if there is no algebraic obstruction (e.g. if $G$ is geometrically simple) then the position of $\log (P)$ in $\Lie (G)\otimes \frac{1}{p}\mathbb{Z}/\mathbb{Z}$ is `random', and hence the probability of lying in a codimension $d$ subspace is $p^{-d}$.

If true, this statement is presumably very difficult to prove. For example, to prove the first case of Fermat's last theorem, it is enough to prove it for primes $p$ such that $\log _p (q)$ is zero mod $p^2$ for all primes $q\leq 113$ \cite{suzuki}, and hence even for $\mathbb{G}_m ^2 $ this would already be a rather strong result.
%
%
%
%
%
\section{The $p$-adic analytic subgroup theorem}
In this section we explain the proof of Theorem \ref{thm1}.

\subsection{A brief review of the Chabauty--Coleman method}
Given an abelian variety $A$ over a number field $K$, and a prime $v|p$ of $K$, we obtain a $v$-adic logarithm
\[
\log _v :A(K_v )\to \Lie (A_{K_v })\simeq H^0 (A_{K_v },\Omega )^* .
\]
If $X$ is a subvariety of $A$, then we deduce that the image of $X(K) \subset X(K_v )$ in $\Lie (A_{K_v })$ is contained in 
$
\log _v (A(K)).
$ When $i:X\to A$ is a curve which generates $A$, and the topological closure of $A(K)$ in $A(K_v )$ is not finite index, this implies that $X(K)$ is contained in the finite set $(\log \circ i)^{-1}(\overline{A(K)}_{\Q _p } )$, where we write $\overline{A(K)}_{\Q _p }$ to mean the $\Q _p $-vector subspace of $H^0 (A,\Omega )^* $ spanned by the image of $A(K)$. Alternatively, the image of $X(K)$ in $H^0 (A_{K_v },\Omega )^* $ is contained in the intersection
\[
(\log \circ i)(X)\cap (\overline{A(K)}_{\Q _p } ).
\]
To prove Theorem 1, we want to show that the points of multiplicity at least 2 do not come from algebraic points of $A$.

\subsection {The $p$-adic analytic subgroup theorem}
The proof of Theorem \ref{thm1} uses the $p$-adic analytic subgroup theorem, whose statement we now explain. Let $A$ and $\log _v $ be as above. Taking the limit over finite extensions $K'|K$, we obtain a continuous homomorphism
\[
\log :A(\overline{\Q }_p )\to \Lie (A(\overline{\Q }_p ))
\]
which may hence be extended to
\[
\log :A(\C _p )\to \Lie (A(\C _p ))
\]

The strongest theorem known about the transcendence properties of the logarithms of algebraic points is the following theorem, which is a $p$-adic analogue of W\"ustholz's analytic subgroup theorem \cite{wustholz}.
\begin{theorem}[Fuchs--Duc Pham \cite{fuchs_pham_subgroup}, Matev \cite{matev2010p}]
Let $G$ be a commutative algebraic group over $\overline{\Q }$ and $V\subset \Lie (G)$ a proper $\overline{\Q }$-subspace. If $g\in G(\overline{\Q })$ has the property that $\log (g)$ is nonzero and lies in $V_{\mathbb{C}_p }\subset \Lie (G)_{\mathbb{C}_p }$, then there is a proper algebraic subgroup $H<G$ defined over $\overline{\Q }$ such that $g\in H(\overline{\Q })$ and $\Lie (H)\subset V$.
\end{theorem}
Now let $X$ as above be a smooth projective curve and take $A=J$ to be its Jacobian variety. Suppose that we have a rational point $b\in X(\Q )$ defining an Abel--Jacobi morphism
\[
i:X\to J.
\]
We suppose that the the $p$-adic closure of $J(\Q )$ in $J(\Q _p )$ is not finite index, and hence that the set $X(\Q _p )_1$ is finite.

Let $x\in X(\Q  )$ and suppose that for all $\omega $ in the space of vanishing differentials, $\int _b \omega $ has a zero of order at least two at $x$. 
Then $\omega $ vanishes at $x$, and hence
\[
\overline{J(\Q )}_{\mathbb{\Q }_p }^{\perp}\subset H^0 (X_{\Q _p },\Omega (x)),
\]
i.e. 
$
\overline{J(\Q )}_{\mathbb{Q}_p }
$
contains the line $H^0 (X_{\Q _p },\Omega (x))^{\perp}$.
Over $\Q $, we have an isogeny 
\[
f=(f_1 ,\ldots ,f_m ):J \to A_1 \times \ldots \times A_m ,
\] 
where for each $i$ either $\overline{A_i (\Q )}$ has finite index in $A(\Q _p )$ or $A_i $ has Mordell--Weil rank at most one. 


The $p$-adic analytic subgroup theorem implies that $\log (A_i (\Q ))$ is either zero or not contained in a proper $\overline{\Q }$-subspace of $\Lie (A_i )_{\C _p }$. For the sake of completeness, we explain in the lemma below why we only need $A_i $ to be simple rather than geometrically simple.

\begin{lemma}\label{easy}
Let $A$ be a simple abelian variety. Suppose $A(\Q )$ is contained in a proper subgroup of $A_{\overline{\Q }}$.Then $A$ has rank zero.
\end{lemma}
\begin{proof}
Since $A(\Q )$ is contained in a proper subgroup of $A_{\overline{\Q }}$, there is a nonzero homomorphism of abelian varieties over $\overline{\Q }$,
\[
f:A_{\overline{\Q }}\to B_{\overline{\Q }}
\]
such that the image of $A(\Q )$ in $B$ is torsion. We may descend $B_{\overline{\Q }}$ and $f$ to a finite Galois extension $K|\Q $, and we define 
\[
g:A\to \Res _{K|\Q }B
\]
to be the map to the Weil restriction. We claim that the image of $A$ in $B$ has rank zero. To see this, it is enough to check that the image of $A(\Q )$ in $B(\Q )$ is finite. Over $K$, we have 
\[
(\Res _{K|\Q }B)_K =\prod _{\sigma \in \Gal (K|\Q )}B^{\sigma }
\]
and 
\[
g_K =\prod _{\sigma \in \Gal (K|\Q )}f^{\sigma }.
\]
The image of $A(\Q )$ under $f^{\sigma }$ is torsion for all $\sigma $. 
\end{proof}

The isogeny $f$ induces a decomposition
\[
\Lie (J)\simeq \prod \Lie (A_i )
\]
of $\Q $-vector spaces and a decomposition
\[
\overline{J(\Q )}_{\Q _p }\simeq \prod \overline{A_i (\Q )}_{\Q _p }
\]
of $\Q _p $-vector spaces. The inclusion
\[
H^0 (X_{\Q _p },\Omega (x))^{\perp} \subset \prod \overline{A_i (\Q )}_{\Q _p }
\]
implies, that, for all $i$, $f_i \circ \AJ$ induces an inclusion
\begin{equation}\label{line_eqn}
(f_i \circ \AJ )^* (H^0 (X_{\Q _p },\Omega (x)))^{\perp} \subset \overline{A_i (\Q )}_{\Q _p }.
\end{equation}
Let $i$ be such that the Mordell--Weil rank of $A_i $ is less than or equal to one, and is less than $\dim (A_i )$. Then $\overline{A_i (\Q )}_{\Q _p }$ is a line, and $A_i$ is simple. Since $\overline{A_i (\Q )}_{\Q _p }$ is a line, the inclusion \eqref{line_eqn} means that either $(f_i \circ \AJ )^* (H^0 (X_{\Q _p },\Omega (x)))^{\perp}$ is zero or it equals $\overline{A_i (\Q )}_{\Q _p }$. Since $A_i$ is simple, Lemma \ref{easy} implies the latter can never happen.

The map $X\to A_i $ factors though a finite map of curves $X\to X'$ such that $X'\to A$ is a closed immersion. The statement that the image of $H^0 (X,\Omega (x))^{\perp}$ in $\Lie (A_i )$ is zero is equivalent to saying that the map on tangent spaces
\[
T_x X \to T_{(f_i \circ \AJ )(x)}A_i
\]
is zero, hence $X\to X'$ is ramified at $x$. 

By our assumption on $A_i$, $X'$ also satisfies the Chabauty--Coleman hypothesis (or is a rank zero elliptic curve). Hence by induction on the genus we see that $X$ either satisfies case (i) or case (ii) of the theorem.

\subsection{An example}
It is easy to come up with an example where all vanishing differentials do vanish at a rational point. Such examples are perhaps well known to experts, but for completeness we give one explicitly. Let
\[
f(x)=x^4 +a_3 x^3 +a_2 x^2 +a_1 x+a_0 
\]
be a separable polynomial.
Let $X_1 ,X_2 ,X_3 $ be the hyperelliptic curves defined by $xf(x),f(x)$ and $f(x^2 )$ respectively. For $i=1,2$ let $f_i $ be the natural map from $X_3 $ to $X_i$. Then $(f_1 ,f_2 )$ induces an isogeny $\Jac (X_3 )\sim \Jac (X_1 )\times \Jac (X_2 )$.

Suppose $\Jac (X_2 )$ has rank zero and $\overline{\Jac (X_1 )(\Q )}$ has finite index in $\Jac (X_1 )(\Q _p )$ (for some fixed prime $p$). Then the spaces of vanishing differentials for $X_3 $ is exactly equal to $f_2 ^* H^0 (X_2 ,\Omega )$, which is spanned by $xdx/y$. Hence all vanishing differentials vanish at $(0,0)$. 

The above example has the feature that the vanishing differentials are actually defined over $\Q $, and come from a rank zero isogeny factor. With a bit more work we can also find an example where all vanishing differentials vanish at a rational point, and the Jacobian does not contain an isogeny factor of rank zero. Moreover it is actually an example of case (2) of Theorem \ref{thm1}, and so although there is a rational point at which all vanishing differentials are zero, by the $p$-adic analytic subgroup theorem the space of vanishing differentials is \textit{not} defined over $\Q $. As far as we are aware, such an example is not explicitly written down in the literature. 

Let $f_1 (x), f_2 (x)$ be monic polynomials, such that $f_1 (x)f_2 (x)$ is separable. For $i=1,2$, let $X_i $ be the hyperelliptic curve given by $y^2 =f_i (x)$. Let $X_3 $ be the curve given by $y^2 =f_1 (x)f_2 (x)$, and let $X_4 $ be the curve given by 
\begin{align*}
& y_1 ^2 =f_1 (x) \\
& y_2 ^2 =f_2 (x). \\
\end{align*}
For $i=1,2,3$, let $g_i $ be the natural surjection from $X_4$ to $X_i $. Then $(g_1 ,g_2 ,g_3 )$ induces an isogeny
\[
\Jac (X_4 )\sim \Jac (X_1 )\times \Jac (X_2 )\times \Jac (X_3 )
\]
(see \cite[Proposition 2.2]{bruin}).
Now suppose that, for $i=2,3$, $\overline{\Jac (X_i )(\Q )}_{\Q _p }=\Lie (\Jac (X_i ))_{\Q _p }$, and that for $i=1$, the genus of $X_i$ is bigger than 1, but $\Jac (X_1 )$ has rank 1. Then all vanishing differentials are in the image of $g_1 ^* $. The morphism $g_1 $ is ramified at the roots of $f_2 $, and hence if $z\in X_3 (\Q )$ has $x$-coordinate which is a root of $f_2 $, then all vanishing differentials lie in $H^0 (X,\Omega (z))$.

We give an example of a suitable choice of polynomials. Take $f_1 =(x^2 -1)(x-a)(x-b)(x-c)$ and $f_2 =(x+a)(x+b)(x+c)$, where $a,b,c$ are pairwise distinct positive rational numbers not equal to $1$. Then the Jacobian of $X_3 $ is isogenous to the Jacobian of $X_5 \times X_6 $, where $X_5 $ is the hyperelliptic curve given by $y^2 =f_3 (x)=x(x-1 )(x-a^2 )(x-b^2 )(x-c^2 )$, and $X_6 $ is the hyperelliptic curve given by $y^2 =f_4 (x)=(x-1 )(x-a^2 )(x-b^2 )(x-c^2 )$. Hence if $(a,b,c)$ is such that
\begin{align*}
& \rk \Jac (X_1 )=1, & rk \Jac (X_2 )\geq 1 \\
& \rk \Jac (X_5 )\geq 2, & \rk \Jac (X_6 )\geq 1 \\
\end{align*}
(and furthermore $\overline{\Jac (X_5 )(\Q )}$ is finite index in $\Jac (X_5 )(\Q _p ))$, and there is an $x_0 \in \{\pm 1,a,b,c\}$ such that $f_2 (x_0 )=z_0 ^2$ is a square, then $P=(x_0 ,0,z_0 )\in X(\Q )$ will be a rational point at which all vanishing differentials vanish. A computation in magma shows that we may take $(a,b,c)=(12,4,-2)$, giving the curve $X$ defined by 
\begin{align*}
y_1 ^2 =(x^2 -1)(x-12)(x-4)(x+2) \\
y_2 ^2 =(x+12)(x+4)(x-2)
\end{align*}
and $P=(4,0,\pm 16)$.
\section{Bounding the proximity of rational points}
In the same way that the proof of Theorem \ref{thm1} followed from a result in transcendence theory controlling when logarithms of algebraic points on abelian varieties can be contained in proper $\Q $-subspaces of their Lie algebras, Theorem \ref{thm2} follows from the following theorem, which controls how close logarithms of algebraic points can get to proper $\Q $-subspaces.
\begin{theorem}[Fuchs, Duc Pham \cite{fuchs_pham}]\label{thmFP}
Let $G$ be an abelian variety over a number field $K$. Fix a projective embedding of $G$ and let $h$ be the corresponding logarithmic height. Let $L$ be a basis of the Lie algebra of $G$. Let 
\[
\ell :\Lie (G)\to K
\]
be a nonzero functional, and let $\delta _L$ be the minimum of the norm of the coefficients of $\ell$ with respect to the basis $L$. Let $B$ be the maximum of the heights of the coefficients of $\ell$, and put $b:=\max \{\log (B),\log (3) \}$.
There is an effectively computable constant $\omega _L >0$ depending on $L$ and there exists an effectively computable positive real constants $c_4$ independent of $b$ and $p$ with the following property: if $u \in  \Lie (G)(\mathbb{C}_p )$ lies in the open polydisk of radius $p^{-\frac{1}{p-1}}|\delta _L |_p $ centred at zero with respect to the basis $L$, and $\exp (u)$ is an algebraic point in $G(K)$, then $l(u)=0$ or
\[
\log |l(u)|_p > -c_4 \omega _L ^{n+3}bh^n (\log b +\log h)^{n+3}\log p,
\]
where $h:=\max \{ h(\exp (u)),\log (3) \}$.
\end{theorem}
This theorem is a $p$-adic analogue of theorems of Hirata--Kohno \cite{HK1} and Gaudron \cite{gaudron} which bound the (Euclidean) distance of the complex logarithm of an algebraic point from a vector subspace of $\Lie (J)_{\mathbb{C}}$ which descends to $\overline{\Q }$.

Hirata--Kohno has used the Archimedean version of this theorem (proved in \cite{HK1}) to bound the height of integral points on curves in terms of the regulator of the Jacobian \cite{HK2}. However, in general it is difficult to use this to control the height of \textit{rational} points because the bounds on the proximity of logarithms of algebraic points are in terms of the heights of the points.

The fundamental difference in the $p$-adic case is that, when the rank is less than the genus, the $p$-adic logarithms of Mordell--Weil generators place non-trivial controls on the $p$-adic  logarithms of \textit{all} algebraic points. Somewhat more precisely, the idea of the application to rational points is that, if we have rational points very close to $x$, then
$
\overline{J(\Q )}_{\mathbb{Q}_p }^{\perp}$ is very close to $H^0 (X,\Omega (x))$, or dually $\overline{J(\Q )}_{\mathbb{Q}_p }$ is very close to $H^0 (X,\Omega (x))^{\perp}$.

We now give the proof of Theorem \ref{thm2}. The proof involves the following elementary Newton Polygon translation between proximity of $y\in X(\Q _p )_1$ to $x\in X(\Q )$ and proximity of the vanishing differential to the subspaces $H^0 (X_{\Q _p },\Omega (-n\cdot x))$.

Let $X_{\mathbb{Z}}$ be the minimal proper regular model of $X$ over $\mathbb{Z}$. The N\'eron model of its Jacobian, which we denote by $J_{\mathbb{Z}}$, is then $\Pic ^0 (X_{\mathbb{Z}})$, by \cite[Theorem 9.5.1]{BLR}. Hence the basis $L$ of $H^1 (X_{\mathbb{Z}},\mathcal{O}_X )$ gives a basis of $\Lie (J_{\mathbb{Z}})$.
\begin{lemma}\label{lem1}
Let $m\geq 0$ be the largest integer such that $\omega $ vanishes to order $\geq m$ at $x$ for all vanishing differentials $\omega $. Suppose $x,y\in X(\Q _p )_1$ are congruent modulo $p^n$. Then, for any vanishing differential $\omega \in H^0 (X_{\Q _p },\Omega )$ which lies in the subspace $H^0 (X_{\mathbb{Z}_p },\Omega )$, the $p$-adic distance of $\omega $ from \\ $H^0 (X_{\Q _p },\Omega (-(m+1)\cdot x))$ is at most $p^{-n}$.
\end{lemma}
\begin{proof}
Without loss of generality we may rescale $\omega $ so that it does not vanish in $H^0 (X_{\F _p },\Omega )$ (i.e. so that it isn't a multiple of $p$) since this can only increase the $p$-adic distance. Let $t$ be an integral parameter at $x$. Write
\[
F:=\int _x ^z \omega =\sum _{n>0} a_n t^n (z).
\]
Since $F$ has a zero of valuation at least $n$, the slope $\leq -n$ part of the Newton polygon of $F$ has length $k\geq 2$. Let $(k,v_p (a_k ))$ denote the endpoint of the slope $\leq -n$ part of the Newton polygon of $F$. Since $dF/dt \in \mathbb{Z}_p [\! [t]\! ]$, we must have 
$
v_p (a_k )\geq -v_p (k),
$ hence $v_p (a_1 )\geq (k-1)n -v_p (k)\geq n$.
\end{proof}
Let $x\in X(\Q )$ be as in Theorem \ref{thm2}. Let $m \geq 0$ be the maximum multiplicity to which all vanishing differentials vanish at $x$. Suppose $y\in X(\Q _p )_1$ is congruent to $x$ modulo $p^N$. then by Lemma \ref{lem1}, any $\mathbb{Z}_p$-integral vanishing differential has distance at most $p^{-N}$ from $H^0 (X_{\Q _p },\Omega (-(m+1)x))$. 
%

Given a $\mathbb{Z}_p$ -submodule $B$ of a $\mathbb{Z}_p$-module $A$, we define $B^{p-\sat }$ to be the $p$-saturation of $B$ in $A$, i.e. the group of $x\in A$ such that $p^n \cdot x \in B$ for some $n\geq 0$).
Let 
\[
B:=(f_1 ^{-1}( H^0 (X_{\Q _p },\Omega (-(m+1)\cdot x)))^{\perp})^{p-\sat } \subset \Lie (A_1 )_{\Z _p }
\] denote the $p$-saturation of $f_1 ^{-1}(H^0 (X_{\Q _p },\Omega (-(m+1)\cdot x)))^{\perp}\cap \Lie (A_1 )_{\mathbb{Z}_p }$ in $\Lie (A_1 )_{\mathbb{Z}_p }$. Let $r_{A_1 ,n}$ and $r_{J,n}$ denote the reduction mod $p^n$ maps
\begin{align*}
& r_{A_1 ,n}:\Lie (A_1 )_{\mathbb{Z}_p }\to \Lie (A_1 )_{\mathbb{Z}_p }\otimes \mathbb{Z}/p^n \mathbb{Z} \\
& r_{J ,n}:\Lie (J )_{\mathbb{Z}_p }\to \Lie (J )_{\mathbb{Z}_p }\otimes \mathbb{Z}/p^n \mathbb{Z}
\end{align*} 
Since $\overline{A_1 (\Q )}$ is a line, we must have 
\begin{equation}\label{do_u_need_this}
r_{A_1 ,N}(f_1 (\overline{J(\Q )}^{p-\sat })) = r_{A_1 ,N}(B).
\end{equation}
Let $C=f_{1*}B \subset \Lie (J)_{\mathbb{Z}_p }$. Then \eqref{do_u_need_this} implies
\begin{equation}\label{eqn:rJN}
r_{J,N}(\overline{J(\Q )}_{\Q _p }\cap \Lie (J)_{\mathbb{Z}_p }) \subset r_{J,N}(C).
\end{equation}
Let $P$ be a point of $A_1 $ of infinite order. 
Hence, if $P$ is in $J ^1 (\Q _p ):=\Ker (J (\Q _p )\to J (\F _p ))$ and has height $h(P)$, then by \eqref{eqn:rJN}, together with Theorem \ref{thmFP}, we have
\[
N <c_4 \omega _L ^{n+3}bh(P)^n (\log b +\log h(P))^{n+3},
\]
where $b$ is the height of the a hyperplane in $\Lie (J)$ containing $C_{\Q }:=f_{1 *} f_1 ^{-1}(H^0 (X,\Omega (-(m+1)\cdot x))^{\perp})$ .

We now bound the height of $C_{\Q }$. For an $e$-dimensional $\Q $-subspace $W$ of $\Lie (J)_{\Q _p }$, let $h(W)$ denote the height of $[W]$ in the Grassmannian of $e$-dimensional subspaces of $\Q ^g$ (using our chosen basis of $H^1 (X,\mathcal{O})$). Note that there is a constant $c$ depending only on $g$ such that, for any subspace $W_1 ,W_2$ of $\Q ^g$, we have $h(W_1 +W_2 )\leq h(W_1 )+h(W_2 )+c$. In particular, we deduce there is a constant $c_1$ depending only on $g$ such that
\[
b\leq h(C_\Q )+c_1 \leq h(H^0 (X_{\Q },\Omega (-(m+1)\cdot x)))^{\perp}+\sum _{i=2}^n h(f_i (\Lie (A_i ))),
\]
since 
\[
C_{\Q }=H^0 (X,\Omega (-(m+1)\cdot x))^{\perp}+\sum _{i=2}^n f_i (\Lie (A_i ))
\]

By the Masser--W\"ustholz tangent space lemma \cite[\S 8]{MW93}, there is an effectively computable constant $c_5$ depending only on $g$ such that
\[
h(f_i (\Lie (A_i )))<c_5 \max \{ 1,h_F (J),\log (d_i )\},
\]
(recall $d_i $ is the degree of $A_i $ in $J$). On the other hand, there is a constant $c_6 $ depending on the choice of basis of $H^1 (X,\mathcal{O})$ and $c_7 $ depending only on $g$ such that
\[
h(H^0 (X,\Omega (-(m+1 )\cdot x)))\leq c_7 h(x)+c_6.
\]
Hence if $P\in J(\Q )\cap J^1 (\Q _p )$ is such that no non-zero multiple of $P$ lies in $\prod _{i=2}^n A_i $, then there are effectively computable constants $c_8, c_9$ such that
\[
N <c_8 \omega _L ^{g+3}(h(x)+c_9 \kappa )h(P)^n (\log (h(x)+c_9 \kappa ) +\log h(P))^{g+3}.
\]

The only remaining point is to bound the minimum height of a point in $J ^1 (\Q _p )\cap J(\Q )$ such that no non-zero multiple of it lies in $\prod _{i=2}^n A_i$ in terms of the heights of Mordell--Weil generators of $J(\Q )$. Our approach is simply to bound the heights of $m_p \cdot P_0$, where $P_0$ is a Mordell--Weil generator of $J$ and $m_p :=\# J (\F _p )$ is the number of $\F _p $-points of the N\'eron model of $A_1 $. By a theorem of Silverman and Tate \cite[Theorem A]{silverman}, there are effectively computable constants $c_{10} ,c_{11}  $ depending only on $g$ such that 
\[
h(m_p \leq \cdot P_0 )\leq m_p ^2 h(P_0 )+c_{10} \cdot h_{\Theta } (J)+c_{11} ,
\]
where $h_{\Theta }$ denotes the theta height of $J$. On the other hand, by Pazuki \cite{pazuki}, there are effectively computable constants $c_{12} ,c_{13}$ such that
\[
h_{\Theta }(J)\leq c_{12} h_F (J)+c_{13} .
\]
Hence we deduce \eqref{eqn:main_ineq}, and Theorem \ref{thm2}.

\section{Wieferich statistics}
We first review the setting and statement of Katz's conjecture on the position of the mod $p$ logarithm of a point on an abelian variety. Let $G$ be an abelian variety over $\mathbb{Z}_p $. Then we have an exact sequence
\[
0\to \Lie (G)\otimes \mathbb{Z}/p\mathbb{Z} \to G(\mathbb{Z}/p^2 \mathbb{Z})\to G(\mathbb{Z}/p\mathbb{Z})\to 0.
\]
Hence, for any $P\in G(\mathbb{Z}_p )$ with image $\overline{P}$ in $G(\mathbb{Z}/p^2 \mathbb{Z})$, we obtain 
\[
W_P (p):=n_p \overline{P}/p \in \frac{1}{p}\Lie (G)/\Lie(G),
\]
where $n_p :=\# G(\F _p )$.
We write $pW_P (p)$ to mean $n_p \overline{P} \in \Lie (G)\otimes \mathbb{Z}/p\mathbb{Z}$.
\begin{lemma}[\cite{poonen_torsion}, Lemma 7]\label{poonen_lemma}
For $P\in G(\mathbb{Z}_p )$  mapping to $\overline{P}\in G(\mathbb{Z}/p^2 \mathbb{Z})$, we have
\[
\frac{n_p }{p}\log _G (P)\equiv pW_P (p) \mod p\Lie (G).
\]
\end{lemma}
%
%
Now let $A$ be an abelian variety over $\Q $.
Let $\mathcal{A}/\mathbb{Z}$ denote the N\'eron model of $G$. The following is a special case of a conjecture of Katz.

\begin{conjecture}[Katz, \cite{katz}]
If $G$ is a geometrically simple group variety, and $P$ is a point of infinite order, then the mod $p$ logarithm $\overline{\log }_p (P)$ is equidistributed in $\Lie (G_{\mathbb{Z}})\otimes \mathbb{R}/\mathbb{Z}$ as $p$ varies over primes of good reduction.
\end{conjecture}

\subsection{Wieferich statistics and effective Mordell}
It is not clear whether Katz's conjecture has direct applications to rational points on curves. However, by Theorem \ref{thm2}, stronger forms of Katz's conjecture, or other more quantitative conjectures about the Wieferich statistics of points on the Jacobian, as we now explain.

%
Let $X$ be a smooth projective geometrically irreducible curve of genus $g>2$ over $\Q$, with Jacobian $J$. Let $A$ be a geometrically simple abelian variety admitting a surjection $\pi :J\to A$ defined over $\Q$. Let $\{ b_1 ,b_2 \} \subset X(\Q )$ be rational points such that $b_1 -b_2 $ is torsion.

\begin{theorem}\label{thm3}
Suppose that the rank of $A$ is less than $\dim (A)-1$. Let $\mathcal{A}$ denote the N\'eron model of $A$. Suppose that both of the following two hypotheses hold.
\begin{enumerate}
\item For any corank two submodule $S$ of $\Lie (\mathcal{A})$, and any point $P\in A(\Q )$ of infinite order, there is an effectively computable constant $c$ such that, for all $p>c$, $W_p (P)$ does not lie in $S$.
\item For any prime $p$, there is an effectively computable lower bound on the $p$-adic distance from a point $x\in X(\Q )-\{ b_1 ,b_2 \}$ to $\{b_1 ,b_2 \}.$
\end{enumerate}
Then there is an effectively computable bound on the height of rational points in $X(\Q )$, in terms of the regulator of the $A$.
\end{theorem}
\begin{proof}
It is enough to show that hypothesis 1 of the theorem give an effectively computable bound on the largest prime $p_0$ such that $x\in X(\Q )-\{ b_1 ,b_2 \}$ is congruent to $b_1 $ or $b_2 $ modulo $p_0$. Indeed, given such a bound, let $f:X\to \mathbb{P}^1$ be a non-constant function with divisor a multiple of $b_1 -b_2 $. Then hypothesis 2 gives a bound on the height of $f(x)$, and hence on the height of $x$.\

To show that hypothesis 1 gives such a bound on $p_0$, note that $X(\Q _p )_1$ containing a point congruent to but not equal to $b_1$ implies that, for every $P\in A(\Q )$, $W_p (P)$ lies in the image of the line $H^0 (X,\Omega (b_1 ))$, which has codimension at least 2 in $\Lie \mathcal({A})$ (and similarly for $b_2$ ).
\end{proof}

%
%
%

\subsection{The relation to Stoll's strong Chabauty conjecture}
Let $X$ be a smooth projective geometrically irreducible curve of genus $g>1$. In \cite{stoll}, Stoll makes the following conjecture.
\begin{conjecture}[``Strong Chabauty conjecture'', Stoll]
Let $f:X\to A$ be a morphism to an abelian variety whose image is not contained in the translate of a proper subgroup. If the rank of $A$ is less than $\dim (A)-1$, then there is a finite subscheme $Z\subset A$ and a set of primes $p$ of density 1 such that the intersection of the topological closure of $A(\Q )$ in $A(\Q _p )$ with $f(X(\Q _p ))$ is contained in $Z(\Q _p )$. 
\end{conjecture}
To relate this to the set $X(\Q _p )_1$ as we have defined it, it is helpful to consider an even stronger conjecture. 
\begin{conjecture}\label{strong_conjecture}
Let $f:X\to A$ be a morphism to an abelian variety whose image is not contained in the translate of a proper subgroup. If the rank of $A$ is less than $\dim (A)-1$, then there is a finite subscheme $Z\subset A$ and a set of primes $p$ of density 1 such that the intersection of $\overline{A(\Q )}^{p-\sat}$ in $A(\Q _p )$ with $f(X(\Q _p ))$ is contained in $Z(\Q _p )$. 
\end{conjecture}
This conjecture is in a sense stronger than hypothesis (1) of Theorem \ref{thm3}, in that it predicts that in the setting of Theorem \ref{thm3}, $X(\Q _p )_1$ has an arithmetic structure for most primes. The fact that it only considers a set of primes of density 1 means it is not enough to imply hypothesis (1) of Theorem \ref{thm3}, but does imply the following weaker version.
\begin{lemma}
Let $X,b$ be as in Theorem \ref{thm2}. Suppose the rank of $J$ is less than $g-1$. Then Conjecture \ref{strong_conjecture} implies that the set of primes $p$ for which $X(\Q _p )_1$ contains a point $x$ congruent to $b$ modulo $p$, but not equal to $b$, has density zero.
\end{lemma}
\begin{proof}
Let $J(\Q _p )^{\sat }$ denote the full saturation of $\overline{J(\Q )}$ in $J(\Q _p )$ -- i.e. the set of $P\in J(\Q _p )$ such that $n\cdot P\in \overline{J(\Q )}$ for some $n>0$. The set $X(\Q _p )_1$ is equal to the the intersection of $\AJ _b (X(\Q _p ))$ with $J(\Q _p )^{\sat }$. Since $J^1 (\Q _p )$ is a $\mathbb{Z}_p$-module, we have
\[
\overline{J(\Q )}^{\sat }\cap J^1 (\Q _p )=\overline{J(\Q )}^{p-\sat }\cap J^1 (\Q _p ).
\]
Hence the set of $x\in X(\Q _p )_1$ congruent to $b\in X(\Q )$ modulo $p$ is equal to the intersection of $X(\Q _p )$ with $\overline{J(\Q )}^{p-\sat }\cap J^1 (\Q _p )$, which Conjecture \ref{strong_conjecture} predicts is equal to $\{ b\}$ for a density one set of primes.
\end{proof}

It is natural to ask how often $\overline{A(\Q )}$ and $\overline{A(\Q )}^{\sat }$ are different. In fact, this question is again related to Wieferich statistics. 
\begin{lemma}
Suppose $P\in \overline{A(\Q )}$ has the property that $\frac{1}{p}\cdot P$ is in $\overline{A(\Q )}^{p-\sat }$ but not in $\overline{A(\Q )}$. Then $pW_P (p)=0$. 
\end{lemma}
\begin{proof}
This follows from the fact that the map $P\mapsto pW_P (p)$ is a homomorphism from $A(\Q _p )$ to an $\F _p $-vector space.
\end{proof}
\section{Possible generalisations to the Chabauty--Kim method}
Let $X$ be a smooth projective geometrically irreducible curve of genus $g>1$ with Jacobian $J$.
If $\overline{J(\Q )}$ is finite index in $J(\Q _p )$, then Chabauty's method no longer gives a control on the rational points of $X$. However, one can sometimes use a non-abelian generalisation of Chabauty--Coleman, due to Kim, instead. In this section we suggest partial analogues of some of the constructions above to Kim's setting. As many of the technical details of Kim's construction are not relevant to this concerns of this paper, we only outline the general theory in broad strokes, and refer the reader to more foundational articles \cite{Kim}, \cite{Kim:Selmer} for more details.

The Chabauty--Kim method instead uses the commutative diagram
\[
\begin{tikzcd}
X(\Q ) \arrow[d] \arrow[r, "j_n"] & \Sel (U_n ) \arrow[d, "\loc _p "] \\
X(\Q _p ) \arrow[r, "j_{n,p}"]           & U_n ^{\dR}/F^0    .           
\end{tikzcd}
\]
Here $U_n $ is a finite dimensional unipotent group over $\Q _p $ with a continuous action of $\Gal (\overline{\Q }|\Q )$, and $U_n ^{\dR}$ is a finite dimensional unipotent group over $\Q _p $ with a subgroup $F^0 $. The Selmer variety $\Sel (U_n )$ is a scheme of finite type over $\Q _p $, and the morphism $\loc _p $ comes from a morphism of algebraic varieties. The map $j_{n,p}$ is locally analytic. We define
\[
X(\Q _p )_n :=j_{n,p}^{-1}\loc _p (\Sel (U_n )).
\]
This contains $X(\Q )$, and Kim conjectures that, for $n\gg 0$, they are equal \cite[Conjecture 3.1]{BDCKW}.

Although it is usually defined as a set, in practice (e.g. when doing computations) $X(\Q _p )_n$ actually has a richer structure of zero dimensional analytic space in the same way that $X(\Q _p )_1$ does. The correct notion of `analytic space' is a little subtle, because the map $j_{n,p}$ is not a morphism of rigid analytic space. However, it is locally analytic, so we may fix an $\F _p $ point $x_0$ and restrict to the space $]x_0 [\subset X_{\Q _p }^{\an }$ of points reducing to $x_0 $. Then $j_{n,p}|_{]x_ 0 [}$ is a morphism of rigid analytic spaces, and hence we may form the fibre product
\[
X(\Q _p )_n \cap ]b_0 [ := ]b_0 [ \times _{(U_n ^{\dR} /F^0 )^{\an }}(j_{n,p}\Sel (U_n ))^{\an }
\]
(where $j_{n,p}\Sel (U_n )$ denotes the scheme-theoretic image). The generalisations of Theorems \ref{thm1} and \ref{thm2} would then involve understanding when $X(\Q _p )_n \cap ]b_0 [$ is non-reduced, and bounding how close its $p$-adic points can get to a given rational point in $X(\Q )$.
This should also have a relation to questions in transcendence theory, however as we are in a `non-linear' situation, it seems that the above theorems in transcendence theory cannot be applied (at least not naively).
%
%
\subsection{Example: integral points on elliptic curves}
It is worthwhile considering the simplest example. Let $E$ be an elliptic curve of Mordell--Weil rank one and Tamagawa number one at all bad primes. Fix a generator $Q$ of $E$, and a prime $p$ of good ordinary reduction. Now suppose we have an integral point $P$ of infinite order on a rank 1 elliptic curve. Then the Chabauty--Kim method implies $E(\mathbb{Z}_p )_2 $ is contained in the zeroes of the Coleman function
\[
f(z)=(\int  ^Q _0 \omega _0 )^2 \int ^z _P \omega _0 \omega _1 -h(Q)(\int ^z _P \omega _0 +\int ^P _{-P}\omega _0 )\int ^z _P \omega _0
\]
where $\omega _i :=x^i dx/2y$ and $h$ is the $p$-adic height associated to the unit root splitting (see \cite{Kim:elliptic}). The derivative evaluated at $P$ is given by
\[
-2h(Q)\log (P) P^* \omega
\]
(the other terms vanish). We deduce the following `non-abelian version' of Theorem \ref{thm1}.
\begin{proposition}
Suppose $f(z)$ has a repeated root at the integral point $P$. Then the $p$-adic height pairing is degenerate, i.e. $h$ is identically zero on rational points.
\end{proposition}
We can similarly obtain a non-abelian version of Theorem 2, in terms of the valuation of the $p$-adic height.

\bibliography{references}
\bibliographystyle{alpha}
\end{document}